\documentclass[10pt]{amsart}

\usepackage{amsmath,amssymb,amsthm,mathtools}
\usepackage{enumitem}
\usepackage{booktabs}
\usepackage{array}
\usepackage{tikz}
\usetikzlibrary{positioning,calc}
\usepackage[colorlinks=true,linkcolor=blue,citecolor=blue,urlcolor=blue]{hyperref}
\hypersetup{
  pdftitle={Higher-Power Inverse Functional Identities and Frobenius Collision Obstructions},
  pdfauthor={Mohsen Aliabadi},
  pdfsubject={Functional identities on division rings},
  pdfkeywords={additive maps, division rings, finite fields, Frobenius maps, functional identities, generalized polynomial identities, characteristic two}
}

\numberwithin{equation}{section}

\newtheorem{theorem}{Theorem}[section]
\newtheorem{lemma}[theorem]{Lemma}
\newtheorem{proposition}[theorem]{Proposition}
\newtheorem{corollary}[theorem]{Corollary}
\newtheorem{problem}[theorem]{Problem}

\theoremstyle{definition}
\newtheorem{definition}[theorem]{Definition}
\newtheorem{example}[theorem]{Example}

\theoremstyle{remark}
\newtheorem{remark}[theorem]{Remark}

\newcommand{\F}{\mathbb F}
\newcommand{\Fp}{\mathbb F_p}
\newcommand{\Fq}{\mathbb F_q}
\newcommand{\Z}{\mathbb Z}
\newcommand{\Q}{\mathbb Q}
\newcommand{\id}{\operatorname{id}}
\newcommand{\charac}{\operatorname{char}}
\newcommand{\Sn}{\mathcal S_n}

\title[Higher-power inverse identities]{Higher-Power Inverse Functional Identities and Frobenius Collision Obstructions}

\author{Mohsen Aliabadi}
\address{Department of Mathematics, Clayton State University, Morrow, GA, USA}
\email{mohsenaliabadi@clayton.edu, mohsenmath88@gmail.com}

\subjclass[2020]{Primary 16R60; Secondary 16K40, 16R50, 12E20, 11T06}
\keywords{Additive maps, division rings, finite fields, Frobenius maps, functional identities, generalized polynomial identities, characteristic two}

\begin{document}

\begin{abstract}
Let $D$ be a division ring, let $n\geq 2$, and let $f,g:D\to D$ be additive maps satisfying
\[
        f(x)x^{-1}+x^n g(x^{-1})=0\qquad (x\in D^\times).
\]
We establish general vanishing criteria, give an explicit classification over finite fields, and determine the additive-polynomial solutions over infinite fields of positive characteristic.  If $\F_q\subseteq Z(D)$, every additive map $D\to D$ admits a canonical decomposition into $\F_q^\times$-weight components, and the identity pairs precisely the weights $r,s$ satisfying $r+s\equiv n+1\pmod{q-1}$.  Consequently, for $q=p^m$, the $\F_q$-vector space $\mathcal S_n(\F_q)$ of solutions over $\F_q$ satisfies
\[
        \dim_{\F_q}\mathcal S_n(\F_q)
        =\#\{(i,j):0\leq i,j<m,\ p^i+p^j\equiv n+1\pmod{q-1}\}.
\]
For infinite fields of characteristic $p>0$, the additive-polynomial solutions are exactly the sums of paired Frobenius terms with $p^i+p^j=n+1$; this statement does not purport to classify all additive maps.  Prime-field dilation gives complete vanishing in characteristic zero and, in characteristic $p>0$, whenever $(p-1)\nmid(n-1)$.

In characteristic two, we prove vanishing for $n=2$ on every noncommutative division ring. If $[D:Z(D)]=\infty$, we also prove vanishing for all $n$ when the center is infinite, and for $2\leq n\leq q-1$ when $Z(D)=\F_q$ is finite. The remaining cases are stated explicitly.
\end{abstract}

\maketitle

\section{Introduction}

Let $D$ be a division ring.  We write $D^\times=D\setminus\{0\}$ and $Z(D)$ for the center of $D$.  A map $h:D\to D$ is called \emph{additive} if
\[
        h(x+y)=h(x)+h(y)\qquad (x,y\in D).
\]
Thus $h$ is a homomorphism of the underlying additive group.  It is not assumed to be multiplicative, $Z(D)$-linear, or semilinear.  In particular, central scalars cannot usually be moved through $h$, except for scalars from the prime field or under additional hypotheses.

Functional identities on rings and division rings are closely related to generalized polynomial identities, derivations, elementary operators, and structural properties of prime and semiprime rings.  We refer to Bre\v sar, Chebotar, and Martindale \cite{BCM} for background on functional identities, and to Lam \cite{Lam} for standard facts on division rings and central simple algebras.

Several recent works study inverse-type functional identities on division rings.  Catalano studied additive maps satisfying identities involving inverses on division rings and simple artinian rings \cite{Catalano}.  Catalano and Merch\'an treated rational identities of the form
\[
        f(x)+x^m g(x^{-1})=0
\]
for particular exponents and characteristic assumptions \cite{CatalanoMerchan}.  Lee and Lin subsequently used generalized polynomial identities to study this family on noncommutative division rings \cite{LeeLin}.  Ero\u glu, Lee, and Lin treated the corresponding characteristic-two problem \cite{ELL}.

Over a \emph{commutative} field, the inverse identity considered in this paper becomes an equation of the form
$F(x)+M(x)G(x^{-1})=0$ with $M(x)=x^{n+1}$.  The general field equation was studied by Ng in characteristic different from two \cite{Ng} and by Kawai and Ferreira in characteristic two \cite{KawaiFerreira}.  Accordingly, our treatment of fields is intended to provide an explicit Frobenius-coefficient description for this power multiplier, rather than to claim a first classification of arbitrary additive solutions over fields.  The noncommutative case, in which the factor on the far right cannot be moved across $g(x^{-1})$, is the main focus of the vanishing results below.

The placement of the factors in the identity studied below is different and essential: right-multiplication produces a terminal factor \(x\) that cannot be commuted past the value of an additive map.  Thus the results for \(f(x)+x^m g(x^{-1})=0\) do not apply directly.  For the identity considered here, Catalano, Leavens, and Merch\'an established vanishing for the two cases
\[
        f(x)x^{-1}+x^2g(x^{-1})=0,
        \qquad
        f(x)x^{-1}+x^3g(x^{-1})=0,
\]
when the characteristic is different from \(2\) and \(3\) \cite{CLM}.  The elementary scaling argument in Corollary~\ref{cor:even-n} also covers characteristic three when \(n=2\); the substantive missing case \(n=2\) in characteristic two is addressed by Theorem~\ref{thm:n2-char-two}.  Thus the \(n=2\) vanishing statement holds over all noncommutative division rings.  The later sections establish higher-power vanishing theorems in characteristic two under explicit hypotheses on the center and the central dimension.

We study the following higher-power inverse identity.

\begin{problem}\label{prob:main}
Let $D$ be a division ring, let $n\geq 2$, and let $f,g:D\to D$ be additive maps satisfying
\begin{equation}\label{eq:main}
        f(x)x^{-1}+x^n g(x^{-1})=0\qquad (x\in D^\times).
\end{equation}
When does one have $f=g=0$?
\end{problem}

The answer depends strongly on characteristic and on commutativity.  Over fields of positive characteristic, Frobenius powers give nonzero solutions.  Therefore a correct general statement must distinguish vanishing mechanisms from Frobenius-type obstructions.  The goal of this paper is to isolate those obstructions explicitly.

A basic warning is that noncommutativity changes the form of the identity.  Multiplying \eqref{eq:main} on the right by $x$ gives
\begin{equation}\label{eq:right-multiplied}
        f(x)=-x^n g(x^{-1})x.
\end{equation}
In a noncommutative division ring this is not the same as
\[
        f(x)=-x^{n+1}g(x^{-1}).
\]
The final factor $x$ in \eqref{eq:right-multiplied} cannot be moved past $g(x^{-1})$.  This is one reason why the finite-field Frobenius examples do not automatically extend to noncommutative division rings.

\subsection*{Main results}

A useful starting point is the character-orthogonality decomposition in Theorem~\ref{thm:canonical-weight-pairing}.  If $\F_q\subseteq Z(D)$, every additive map decomposes into $\F_q^\times$-weights, without any semilinearity hypothesis, and the identity couples precisely the weights $r$ and $s$ satisfying
\[
        r+s\equiv n+1\pmod{q-1}.
\]
Specializing to $D=\F_q$, where $q=p^m$, gives the explicit coefficientwise finite-field classification in Corollary~\ref{cor:finite-field-collision-graph}.  Let $\Sn(\Fq)$ be the $\Fq$-vector space of solutions, and let $\Gamma_n(q)$ be the bipartite graph joining $i$ to $j$ when
\[
        p^i+p^j\equiv n+1\pmod{q-1}.
\]
Every vertex has degree at most one, and
\[
        \dim_{\Fq}\Sn(\Fq)=|E(\Gamma_n(q))|\leq m.
\]
Thus the finite-field obstruction is exactly a Frobenius collision condition.  The prime-field scaling argument of Theorem~\ref{thm:prime-field-scaling} further shows that if there exists $\lambda$ in the prime field of $D$ such that $\lambda^{n-1}\neq 1$, then every additive solution is zero.  In particular, characteristic zero gives automatic vanishing.

The principal noncommutative results concern characteristic two, where the elementary scaling choice $\lambda=-1$ is unavailable.  A Hua-type argument shows that, over a division ring infinite-dimensional over its center, every solution reduces to a one-map identity $h=f=g$.  The first new ingredient after this reduction is an inverse-free identity expressing $h(x^2)$ in terms of $h(x)$ and $h(x+1)$.  For $n=2$, its linearization yields a cocycle identity involving additive commutators; Theorem~\ref{thm:n2-char-two} then proves that every solution over a noncommutative division ring is zero.  For general $n$, central rescaling forces
\[
        h(\lambda^2x)=\lambda^{n+1}h(x)
        \qquad (\lambda\in Z(D)^\times).
\]
If $Z(D)$ is infinite, this relation proves complete vanishing: directly when $n+1$ is not a power of two, and, in the power-of-two case, after the inverse-free identity is used to show that $h$ is central-valued.  Corollary~\ref{cor:char-two-infinite-center} therefore gives $f=g=0$ for every $n\geq2$ whenever $[D:Z(D)]=\infty$ and $Z(D)$ is infinite.

The finite-center case requires a different separation argument.  When $Z(D)=\F_q$ and $[D:Z(D)]=\infty$, we rewrite the inverse-free identity as a sum of homogeneous brackets and separate their degrees modulo $q-1$.  For $2\leq n\leq q-1$, the degree-one bracket is isolated.  It forces either immediate vanishing or central-valuedness; in the latter case, a Frobenius-root argument and the bounded-degree theorem finish the proof.  Theorem~\ref{T:main} consequently gives complete vanishing whenever
\[
        Z(D)=\F_q,\qquad [D:Z(D)]=\infty,\qquad 2\leq n\leq q-1.
\]

\section{Preliminaries}

\begin{definition}
The \emph{prime field} of a division ring $D$ is the smallest subfield of $D$.  It is isomorphic to $\Q$ if $\charac D=0$, and to $\Fp$ if $\charac D=p>0$.  Since it is generated by $1_D$, it is contained in $Z(D)$.
\end{definition}

\begin{lemma}\label{lem:prime-field-linearity}
Let $D$ be a division ring with prime field $P$, and let $h:D\to D$ be additive.  Then
\[
        h(\lambda x)=\lambda h(x)
\]
for every $\lambda\in P$ and every $x\in D$.
\end{lemma}

\begin{proof}
If $\charac D=p>0$, then $P\cong \Fp$, and the assertion follows from additivity.  If $\charac D=0$, then $P\cong\Q$.  Additivity gives $h(mx)=mh(x)$ for every integer $m$.  If $\lambda=a/b\in\Q$ with $b\neq0$, then
\[
        b h(\lambda x)=h(b\lambda x)=h(ax)=a h(x)=b\lambda h(x).
\]
The additive group of $D$ is torsion-free in characteristic zero, so $h(\lambda x)=\lambda h(x)$.
\end{proof}

\begin{definition}\label{def:gpi-algebra}
Let $D$ be a division ring with center $Z=Z(D)$.  The \emph{generalized polynomial algebra} over $D$ in noncommuting variables $X_1,\ldots,X_r$ is the free product of $Z$-algebras
\[
        D *_Z Z\langle X_1,\ldots,X_r\rangle,
\]
where the two copies of $Z$ are identified.  Equivalently, its elements are finite $Z$-linear combinations of words in which coefficients from $D$ and the variables $X_i$ alternate, subject only to the relations already holding in $D$ and to the centrality of $Z$.
\end{definition}

\begin{definition}\label{def:gpi}
A nonzero element $F\in D *_Z Z\langle X_1,\ldots,X_r\rangle$ is called a \emph{generalized polynomial identity}, or \emph{GPI}, for $D$ if
\[
        F(a_1,\ldots,a_r)=0
\]
for every $a_1,\ldots,a_r\in D$.
\end{definition}

\begin{theorem}\label{thm:martindale}
If a division ring $D$ satisfies a nontrivial generalized polynomial identity over its center, then $D$ is finite-dimensional over $Z(D)$.
\end{theorem}

\begin{proof}
This is the division-ring case of Martindale's theorem on prime rings satisfying generalized polynomial identities; see \cite{Martindale}.  We also use the standard fact that the extended centroid of a division ring is its center; see \cite[Chapter~2]{BMM}.
\end{proof}

We shall also use the following standard linear-independence lemma for elementary operators; see, for example, \cite{BMM}.

\begin{lemma}\label{lem:elementary-operators}
Let $D$ be a division ring with center $Z$.  Suppose $a_1,\ldots,a_r\in D$ are linearly independent over $Z$ and $b_1,\ldots,b_r\in D$ satisfy
\[
        \sum_{i=1}^r a_i x b_i=0
        \qquad (x\in D).
\]
Then $b_1=\cdots=b_r=0$.
\end{lemma}

We also recall the bounded-degree theorem for division algebras.

\begin{theorem}[Jacobson \cite{JacobsonBounded}]\label{thm:bounded-degree}
Let $D$ be a division ring with center $Z$.  If there is an integer $d\geq1$ such that every element of $D$ is algebraic over $Z$ of degree at most $d$, then $D$ is finite-dimensional over $Z$.
\end{theorem}

\section{Central weights and Frobenius solutions}

We first record the basic positive-characteristic obstruction over fields and then use the standard orthogonality of multiplicative characters to decompose additive maps over finite central subfields.  This yields an explicit description of the solution spaces for the particular power multiplier in \eqref{eq:main}, complementary to the general field-equation literature \cite{Ng,KawaiFerreira}.

\begin{example}\label{ex:frobenius-basic}
Let $F$ be a field of characteristic $p>0$, and suppose
\[
        n+1=p^a+p^b
\]
for some integers $a,b\geq0$.  For any $c\in F$, define
\[
        g(x)=c x^{p^a},
        \qquad
        f(x)=-c x^{p^b}.
\]
Then $f$ and $g$ are additive maps $F\to F$.  For $x\in F^\times$,
\[
        f(x)x^{-1}+x^n g(x^{-1})
        =-c x^{p^b-1}+c x^{n-p^a}=0,
\]
because $p^b-1=n-p^a$.  Thus nonzero Frobenius solutions exist whenever $n+1$ is a sum of two $p$-powers.
\end{example}

The finite-field calculation below shows that these Frobenius collisions, read modulo $q-1$, account for all solutions over finite fields.  We first establish the central-weight decomposition that organizes this calculation.

\subsection{Canonical weights over finite central subfields}

Let $D$ have characteristic $p>0$ and let $\F_q\subseteq Z(D)$.  We identify a residue $r\in\Z/(q-1)\Z$ with any integer representative when it occurs as an exponent.

\begin{lemma}\label{lem:canonical-weight-decomposition}
Every additive map $u:D\to D$ has a unique decomposition
\[
        u=\sum_{r\in\Z/(q-1)\Z}u_r,
\]
where
\begin{equation}\label{eq:weight-projector}
        u_r(x)=-\sum_{\lambda\in\F_q^\times}
        \lambda^{-r}u(\lambda x)
        \qquad (x\in D)
\end{equation}
and
\begin{equation}\label{eq:weight-homogeneity}
        u_r(\lambda x)=\lambda^r u_r(x)
        \qquad (\lambda\in\F_q^\times,\ x\in D).
\end{equation}
In particular,
\[
        \sum_r u_r=-(q-1)u=u,
\]
where the last equality holds because the element $(q-1)1_D$ equals $-1_D$ in characteristic $p$.
\end{lemma}

\begin{proof}
Each $u_r$ is additive, and a change of variable $\mu=\lambda\lambda_0$ in \eqref{eq:weight-projector} gives
\[
        u_r(\lambda_0x)=\lambda_0^r u_r(x).
\]
The character orthogonality relation
\[
        \sum_{r\in\Z/(q-1)\Z}\lambda^{-r}
        =
        \begin{cases}
        q-1,&\lambda=1,\\
        0,&\lambda\neq1
        \end{cases}
\]
gives $\sum_r u_r=-(q-1)u=u$.  The same relation shows that the operator in \eqref{eq:weight-projector} fixes a component of weight $r$ and annihilates every component of a different weight; hence the decomposition is unique.
\end{proof}

\begin{theorem}\label{thm:canonical-weight-pairing}
Let $D$ be a division ring of characteristic $p>0$ with $\F_q\subseteq Z(D)$, and let $f,g:D\to D$ be additive maps satisfying \eqref{eq:main}.  Write their canonical decompositions as
\[
        f=\sum_{r\in\Z/(q-1)\Z}f_r,
        \qquad
        g=\sum_{s\in\Z/(q-1)\Z}g_s.
\]
Then, for every $r\in\Z/(q-1)\Z$ and $x\in D^\times$,
\begin{equation}\label{eq:canonical-weight-pairing}
        f_r(x)x^{-1}
        +x^n g_{n+1-r}(x^{-1})=0.
\end{equation}
Thus the original identity decomposes canonically into independent identities pairing precisely the weights $r$ and $s$ for which
\[
        r+s\equiv n+1\pmod{q-1}.
\]
\end{theorem}

\begin{proof}
Fix $x\in D^\times$ and apply \eqref{eq:main} to $\lambda x$ for $\lambda\in\F_q^\times$.  Centrality of $\lambda$ and \eqref{eq:weight-homogeneity} give
\[
        \sum_r\lambda^{r-1}f_r(x)x^{-1}
        +\sum_s\lambda^{n-s}x^n g_s(x^{-1})=0.
\]
The characters $\lambda\mapsto\lambda^t$ of the cyclic group $\F_q^\times$ are linearly independent over $\F_q$.  Equating the coefficient of the character $\lambda^{r-1}$ yields \eqref{eq:canonical-weight-pairing}, since $n-s\equiv r-1$ exactly when $s\equiv n+1-r$ modulo $q-1$.
\end{proof}

\subsection{Finite fields and the collision graph}\label{sec:finite-fields}

Let $q=p^m$.  Let $\Sn(\Fq)$ denote the set of all pairs $(f,g)$ of additive maps $\Fq\to\Fq$ satisfying \eqref{eq:main}.  Since $\Fq$ is commutative, pointwise $\Fq$-scalar multiples of solutions are again solutions; hence $\Sn(\Fq)$ is an $\Fq$-vector space.

\begin{lemma}\label{lem:linearized}
Every additive map $h:\Fq\to\Fq$ is represented uniquely in the form
\[
        h(x)=\sum_{i=0}^{m-1}c_i x^{p^i},
        \qquad c_i\in\Fq.
\]
\end{lemma}

\begin{proof}
An additive map $h:\Fq\to\Fq$ is precisely an $\Fp$-linear endomorphism of the $m$-dimensional $\Fp$-vector space $\Fq$.  It is standard that every such endomorphism is represented by a linearized polynomial
\[
        L(x)=\sum_{i=0}^{m-1}c_i x^{p^i};
\]
see \cite[Chapter 3]{LidlNiederreiter}.  For uniqueness, suppose $\sum_{i=0}^{m-1}c_i x^{p^i}=0$ for every $x\in\Fq$.  If the polynomial is nonzero, its degree is at most $p^{m-1}<q$, so it has fewer than $q$ roots, a contradiction.
\end{proof}

\begin{definition}\label{def:collision-graph}
Let $q=p^m$ and $n\geq2$.  The \emph{Frobenius collision graph} $\Gamma_n(q)$ is the bipartite graph with left vertex set
\[
        L=\{0,1,\ldots,m-1\}
\]
corresponding to the coefficients of $g$, and right vertex set
\[
        R=\{0,1,\ldots,m-1\}
\]
corresponding to the coefficients of $f$.  We join $i\in L$ to $j\in R$ if
\begin{equation}\label{eq:collision-congruence}
        p^i+p^j\equiv n+1\pmod {q-1}.
\end{equation}
\end{definition}

\begin{corollary}\label{cor:finite-field-collision-graph}
Let $q=p^m$ and $n\geq2$.  Every vertex of $\Gamma_n(q)$ has degree at most one; equivalently, $\Gamma_n(q)$ is a matching.  Moreover,
\[
\begin{aligned}
\dim_{\Fq}\Sn(\Fq)&=|E(\Gamma_n(q))|\\
&=\#\bigl\{(i,j):0\leq i,j\leq m-1,\;\\
&\hspace{3em}p^i+p^j\equiv n+1\pmod{q-1}\bigr\}.
\end{aligned}
\]
More explicitly, write
\[
        g(x)=\sum_{i=0}^{m-1}a_i x^{p^i},
        \qquad
        f(x)=\sum_{j=0}^{m-1}b_j x^{p^j}.
\]
Then $(f,g)\in\Sn(\Fq)$ if and only if the following conditions hold:
\begin{enumerate}[label=\textup{(\roman*)}]
\item if a left vertex $i$ is isolated in $\Gamma_n(q)$, then $a_i=0$;
\item if a right vertex $j$ is isolated in $\Gamma_n(q)$, then $b_j=0$;
\item if $i\in L$ and $j\in R$ are joined by an edge, then $a_i+b_j=0$.
\end{enumerate}
Consequently, each edge contributes exactly one free scalar.
\end{corollary}

\begin{proof}
By Lemma \ref{lem:linearized}, write
\[
        f(x)=\sum_{j=0}^{m-1}b_j x^{p^j},
        \qquad
        g(x)=\sum_{i=0}^{m-1}a_i x^{p^i}.
\]
The residues $p^0,p^1,\ldots,p^{m-1}$ are pairwise distinct modulo $q-1$.  Indeed, if $0\leq i<j\leq m-1$ and $p^i\equiv p^j\pmod{q-1}$, then, since $\gcd(p,q-1)=1$, cancellation of $p^i$ gives $p^{j-i}\equiv1\pmod{q-1}$.  But $0<j-i<m$, so
\[
        0<p^{j-i}-1<p^m-1=q-1,
\]
which is impossible.

Hence the residues $p^j-1$ are pairwise distinct, and the residues $n-p^i$ are pairwise distinct, modulo $q-1$.  Therefore a term from the $f$-sum can collide with at most one term from the $g$-sum, and conversely.  Such a collision is exactly
\[
        p^j-1\equiv n-p^i\pmod {q-1},
\]
or equivalently \eqref{eq:collision-congruence}.  Thus every vertex of $\Gamma_n(q)$ has degree at most one.

The monomial map $x\mapsto c x^{p^i}$ has $\Fq^\times$-weight $p^i$.  Theorem \ref{thm:canonical-weight-pairing} therefore pairs the coefficient $b_j$ of $f$ with the coefficient $a_i$ of $g$ exactly when
\[
        p^i+p^j\equiv n+1\pmod{q-1}.
\]
An edge $(i,j)$ gives $a_i+b_j=0$, whereas a coefficient attached to an isolated vertex is zero.  Each edge consequently contributes one free scalar and isolated vertices contribute none.
\end{proof}

\begin{corollary}\label{cor:finite-field-dimension-bound}
Let $q=p^m$.  Then
\[
        \dim_{\Fq}\Sn(\Fq)\leq m=\log_p q.
\]
The bound is sharp; equality holds, for example, when $(q,n)=(4,2)$.
\end{corollary}

\begin{proof}
The graph $\Gamma_n(q)$ is a matching on two vertex classes of size $m$, so it has at most $m$ edges.  The equality example is computed below.
\end{proof}

\begin{remark}\label{rem:partial-involution}
The edge relation is symmetric in $i$ and $j$.  Hence it defines a partial map $\sigma$ on $\{0,\ldots,m-1\}$ by $\sigma(i)=j$ whenever
\[
        p^i+p^j\equiv n+1\pmod{q-1}.
\]
The uniqueness proved above shows that $\sigma$ is a partial involution.  Its fixed points are exactly the indices $i$ satisfying
\[
        2p^i\equiv n+1\pmod{q-1}.
\]
\end{remark}

\begin{corollary}\label{cor:nonzero-finite-field}
Let $q=p^m$ and $n\geq2$.  Then $\Sn(\Fq)\neq0$ if and only if there exist $0\leq i,j\leq m-1$ such that
\[
        p^i+p^j\equiv n+1\pmod {q-1}.
\]
\end{corollary}

\begin{example}\label{ex:F4}
Let $q=4$, so $p=2$, $m=2$, and $q-1=3$.  Take $n=2$.  Then $n+1\equiv0\pmod3$, and
\[
        2^i+2^j\equiv0\pmod3
\]
holds exactly for $(i,j)=(0,1)$ and $(i,j)=(1,0)$.  Hence $\dim_{\F_4}\Sn(\F_4)=2$.

The collision graph is shown in Figure \ref{fig:F4-collision}.

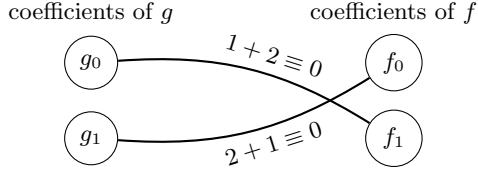
\begin{figure}[h]
\centering
\begin{tikzpicture}[scale=1, every node/.style={font=\small}]
\node[draw,circle,minimum size=7mm] (g0) at (0,1) {$g_0$};
\node[draw,circle,minimum size=7mm] (g1) at (0,0) {$g_1$};
\node[draw,circle,minimum size=7mm] (f0) at (4,1) {$f_0$};
\node[draw,circle,minimum size=7mm] (f1) at (4,0) {$f_1$};
\draw[thick] (g0) to[bend left=18] node[above,sloped,pos=.58] {$1+2\equiv0$} (f1);
\draw[thick] (g1) to[bend right=18] node[below,sloped,pos=.58] {$2+1\equiv0$} (f0);
\node at (0,1.65) {coefficients of $g$};
\node at (4,1.65) {coefficients of $f$};
\end{tikzpicture}
\caption{The collision graph $\Gamma_2(4)$ is a two-edge matching.  Each edge contributes one free scalar to $\Sn(\F_4)$.}
\label{fig:F4-collision}
\end{figure}

Explicitly, every solution has the form
\[
        g(x)=a_0x+a_1x^2,
        \qquad
        f(x)=a_1x+a_0x^2,
        \qquad a_0,a_1\in\F_4,
\]
because the two edge relations are $a_0+b_1=0$ and $a_1+b_0=0$, and $\charac\F_4=2$.
\end{example}

\begin{example}\label{ex:F4-modular-genuine}
The modulus in Corollary \ref{cor:finite-field-collision-graph} is essential, rather than cosmetic.  Take $n=5$ over $\F_4$.  Then $n+1=6$ is not equal to $2^i+2^j$ for $0\leq i,j\leq1$, but
\[
        6\equiv0\equiv 2^0+2^1\pmod3.
\]
Consequently, for arbitrary $a_0,a_1\in\F_4$, the maps
\[
        f(x)=a_1x+a_0x^2,
        \qquad
        g(x)=a_0x+a_1x^2
\]
satisfy \eqref{eq:main}.  Directly, this also follows from $x^5=x^2$ on $\F_4^\times$.
\end{example}

\subsection{Infinite fields and additive polynomial maps}

\begin{definition}\label{def:additive-polynomial}
Let $F$ be a field of characteristic $p>0$.  An \emph{additive polynomial map} $h:F\to F$ is a map of the form
\[
        h(x)=\sum_{i=0}^r c_i x^{p^i},
        \qquad c_i\in F.
\]
\end{definition}

\begin{theorem}\label{thm:infinite-field-additive-polynomial}
Let $F$ be an infinite field of characteristic $p>0$, let $n\geq2$, and let $f,g:F\to F$ be additive polynomial maps.  Write
\[
        g(x)=\sum_{i\in A}a_i x^{p^i},
        \qquad
        f(x)=\sum_{j\in B}b_j x^{p^j},
\]
where $A,B$ are finite subsets of $\Z_{\geq0}$ and all displayed coefficients are nonzero.  Then $(f,g)$ satisfies \eqref{eq:main} if and only if:
\begin{enumerate}[label=\textup{(\roman*)}]
\item for every $i\in A$, there is a unique $j\in B$ such that $p^i+p^j=n+1$;
\item for every $j\in B$, there is a unique $i\in A$ such that $p^i+p^j=n+1$;
\item whenever $p^i+p^j=n+1$, one has $a_i+b_j=0$.
\end{enumerate}
In particular, a nonzero additive-polynomial solution exists if and only if $n+1=p^i+p^j$ for some $i,j\geq0$.
\end{theorem}

\begin{proof}
Substituting the displayed expressions into \eqref{eq:main} gives, for every $x\in F^\times$,
\[
        \sum_{j\in B}b_j x^{p^j-1}
        +
        \sum_{i\in A}a_i x^{n-p^i}=0.
\]
This is a Laurent polynomial identity on $F^\times$.  Since $F$ is infinite, multiplying by a sufficiently large power of $x$ gives an ordinary polynomial over $F$ with infinitely many roots; hence the Laurent polynomial is identically zero.

The exponents $p^j-1$ are pairwise distinct, and the exponents $n-p^i$ are pairwise distinct.  Thus the only possible cancellations are between $p^j-1$ and $n-p^i$, which occurs exactly when $p^i+p^j=n+1$.  For each such pair the coefficient equation is $a_i+b_j=0$.  Any coefficient not involved in such a collision must vanish.
\end{proof}

\begin{remark}\label{rem:infinite-field-additive}
Let $F$ be an infinite field of characteristic $p>0$, such as
$\overline{\Fp}$ or $\Fp(t)$.  An arbitrary additive map $F\to F$
is merely $\Fp$-linear; it need not admit a representation by a
\emph{finite} additive polynomial.  Hence
Theorem~\ref{thm:infinite-field-additive-polynomial} classifies
only the additive-polynomial solutions of \eqref{eq:main}.

The more general equation
\[
        U(x)+M(x)V(x^{-1})=0\qquad(x\in F^\times),
\]
with additive $U,V$ and a multiplicative map $M$, was treated by
Ng for fields of characteristic different from two \cite{Ng} and
by Kawai and Ferreira for fields of characteristic two
\cite{KawaiFerreira}.  Indeed, in a commutative field,
\eqref{eq:main} is equivalent to
\[
        f(x)+x^{n+1}g(x^{-1})=0.
\]
Our theorem gives the explicit Frobenius-support and coefficient
conditions in the subclass of finite additive polynomials; it is
not asserted as a classification of arbitrary additive maps on
infinite fields.
\end{remark}

\section{Scaling and central support}\label{sec:scaling-support}

\subsection{Scaling and semilinear vanishing}

We now turn to arbitrary division rings.  The arguments in this section do not use commutative polynomial expansion.  They use only additivity with respect to suitable central scalars.

\begin{theorem}\label{thm:prime-field-scaling}
Let $D$ be a division ring, let $n\geq2$, and let $f,g:D\to D$ be additive maps satisfying \eqref{eq:main}.  Let $P$ be the prime field of $D$.  Suppose there exists $\lambda\in P^\times$ such that
\[
        \lambda^{n-1}\neq1.
\]
Then $f=g=0$.
\end{theorem}

\begin{proof}
Fix $x\in D^\times$ and set
\[
        A=f(x)x^{-1},
        \qquad
        B=x^n g(x^{-1}).
\]
The identity at $x$ gives $A+B=0$.

Apply the identity at $\lambda x$.  Since $\lambda\in P\subseteq Z(D)$,
\[
        (\lambda x)^{-1}=\lambda^{-1}x^{-1},
        \qquad
        (\lambda x)^n=\lambda^n x^n.
\]
By Lemma \ref{lem:prime-field-linearity},
\[
        f(\lambda x)=\lambda f(x),
        \qquad
        g(\lambda^{-1}x^{-1})=\lambda^{-1}g(x^{-1}).
\]
Thus the identity at $\lambda x$ becomes
\[
        A+\lambda^{n-1}B=0.
\]
Subtracting $A+B=0$ gives $(\lambda^{n-1}-1)B=0$.  Hence $B=0$, and then $A=0$.  Therefore $f(x)=0$ and $g(x^{-1})=0$ for every $x\in D^\times$.  Since additive maps vanish at $0$, we get $f=g=0$.
\end{proof}

\begin{corollary}\label{cor:char-zero}
Let $D$ be a division ring of characteristic zero, and let $n\geq2$.  If additive maps $f,g:D\to D$ satisfy \eqref{eq:main}, then $f=g=0$.
\end{corollary}

\begin{proof}
Take $\lambda=2\in\Q$ in Theorem \ref{thm:prime-field-scaling}.
\end{proof}

\begin{corollary}\label{cor:positive-scaling}
Let $D$ be a division ring of characteristic $p>0$, and let $n\geq2$.  If
\[
        p-1\nmid n-1,
\]
then every additive solution of \eqref{eq:main} is zero.
\end{corollary}

\begin{proof}
The group $\Fp^\times$ is cyclic of order $p-1$.  If $p-1\nmid n-1$, choose $\lambda\in\Fp^\times$ with $\lambda^{n-1}\neq1$ and apply Theorem \ref{thm:prime-field-scaling}.
\end{proof}

\begin{remark}\label{rem:coherence-finite-scaling}
The finite-field classification agrees exactly with Corollary \ref{cor:positive-scaling}.  Indeed, if a solution over $\F_q$, where $q=p^m$, is nonzero, then Corollary \ref{cor:finite-field-collision-graph} gives
\[
        p^i+p^j\equiv n+1\pmod{q-1}
\]
for some $i,j$.  Since $p-1$ divides $q-1$, reduction modulo $p-1$ gives $2\equiv n+1$, and hence $p-1\mid n-1$.
\end{remark}

\begin{corollary}\label{cor:even-n}
Let $D$ be a division ring of characteristic different from two, and let $n\geq2$ be even.  Then every additive solution of \eqref{eq:main} is zero.
\end{corollary}

\begin{proof}
Take $\lambda=-1$ in the prime field and apply Theorem \ref{thm:prime-field-scaling}.
\end{proof}

\begin{definition}\label{def:semilinear}
Let $K\subseteq Z(D)$ be a central subfield, and let $\alpha:K\to K$ be a field endomorphism.  An additive map $h:D\to D$ is called \emph{left $\alpha$-semilinear over $K$} if
\[
        h(\lambda x)=\alpha(\lambda)h(x)
        \qquad (\lambda\in K,\ x\in D).
\]
\end{definition}

\begin{theorem}\label{thm:semilinear}
Let $D$ be a division ring, let $K\subseteq Z(D)$ be a central subfield, and let $n\geq2$.  Let $f,g:D\to D$ be additive maps satisfying \eqref{eq:main}.  Assume that $f$ is left $\beta$-semilinear over $K$ and that $g$ is left $\alpha$-semilinear over $K$, where $\alpha,\beta:K\to K$ are field endomorphisms.  If there exists $\lambda\in K^\times$ such that
\[
        \beta(\lambda)\alpha(\lambda)\neq\lambda^{n+1},
\]
then $f=g=0$.
\end{theorem}

\begin{proof}
Fix $x\in D^\times$ and set $A=f(x)x^{-1}$ and $B=x^n g(x^{-1})$.  Then $A+B=0$.  Applying \eqref{eq:main} at $\lambda x$ gives
\[
        \beta(\lambda)\lambda^{-1}A
        +
        \lambda^n\alpha(\lambda)^{-1}B=0.
\]
Using $A=-B$ gives
\[
        \left(\lambda^n\alpha(\lambda)^{-1}-\beta(\lambda)\lambda^{-1}\right)B=0.
\]
Multiplying the displayed central coefficient by the nonzero scalar $\lambda\alpha(\lambda)$ shows that it is nonzero exactly when $\lambda^{n+1}\neq\beta(\lambda)\alpha(\lambda)$.  Hence $B=0$, then $A=0$, and the usual argument gives $f=g=0$.
\end{proof}

\begin{corollary}\label{cor:central-linear}
Let $K\subseteq Z(D)$ be a central subfield.  Suppose $f,g:D\to D$ are left $K$-linear additive maps satisfying \eqref{eq:main}.  If there exists $\lambda\in K^\times$ with $\lambda^{n-1}\neq1$, then $f=g=0$.  In particular, if $K$ is infinite, then $f=g=0$ for every $n\geq2$.
\end{corollary}

\begin{proof}
Take $\alpha=\beta=\id_K$ in Theorem \ref{thm:semilinear}.  If $K$ is infinite, the polynomial $T^{n-1}-1$ cannot vanish on all of $K^\times$.
\end{proof}

\begin{corollary}\label{cor:frobenius-semilinear}
Assume $\charac D=p>0$, and let $K\subseteq Z(D)$ be a central subfield.  Let $a,b\geq0$.  Suppose $f,g:D\to D$ satisfy
\[
        f(\lambda x)=\lambda^{p^b}f(x),
        \qquad
        g(\lambda x)=\lambda^{p^a}g(x)
\]
for all $\lambda\in K$ and $x\in D$.  If there exists $\lambda\in K^\times$ such that
\[
        \lambda^{p^a+p^b}\neq\lambda^{n+1},
\]
then every additive solution of \eqref{eq:main} is zero.  In particular, if $K$ is infinite and $p^a+p^b\neq n+1$, then $f=g=0$.
\end{corollary}

\begin{proof}
Apply Theorem \ref{thm:semilinear} with $\alpha(\lambda)=\lambda^{p^a}$ and $\beta(\lambda)=\lambda^{p^b}$.
\end{proof}

\subsection{Finite sums of homogeneous components}

The preceding argument separates two terms by comparing their behavior under central dilations.  The same idea gives a support principle for finite sums of homogeneous components.

\begin{lemma}\label{lem:laurent-independence-infinite}
Let $K$ be an infinite field, let $V$ be a vector space over $K$, and let $r_1,\ldots,r_s$ be distinct integers.  If $v_1,\ldots,v_s\in V$ satisfy
\[
        \sum_{\ell=1}^s\lambda^{r_\ell}v_\ell=0
\]
for infinitely many $\lambda\in K^\times$, then $v_1=\cdots=v_s=0$.
\end{lemma}

\begin{proof}
Choose a basis of $V$ over $K$ and take coordinates.  Each coordinate gives a scalar Laurent polynomial identity over $K$.  Multiplying by a sufficiently large power of $\lambda$ gives an ordinary polynomial over $K$ vanishing on infinitely many elements.  Hence every coordinate polynomial is zero.
\end{proof}

\begin{theorem}\label{thm:homogeneous-infinite}
Let $D$ be a division ring of characteristic $p>0$, let $K\subseteq Z(D)$ be an infinite central subfield, and let $n\geq2$.  Suppose additive maps $f,g:D\to D$ satisfy \eqref{eq:main}.  Assume
\[
        f=\sum_{b\in B}f_b,
        \qquad
        g=\sum_{a\in A}g_a,
\]
where $A$ and $B$ are finite subsets of $\Z_{\geq0}$, and
\[
        f_b(\lambda x)=\lambda^{p^b}f_b(x),
        \qquad
        g_a(\lambda x)=\lambda^{p^a}g_a(x)
\]
for all $\lambda\in K$ and all relevant $a,b$.  Then, for every $x\in D^\times$ and every integer $r$,
\[
        \sum_{\substack{b\in B\\ p^b-1=r}} f_b(x)x^{-1}
        +
        \sum_{\substack{a\in A\\ n-p^a=r}} x^n g_a(x^{-1})
        =0.
\]
In particular, an $f_b$-component and a $g_a$-component can cancel only when
\[
        p^a+p^b=n+1.
\]
If no pair $(a,b)\in A\times B$ satisfies this equality, then $f=g=0$.
\end{theorem}

\begin{proof}
Apply \eqref{eq:main} to $\lambda x$, with $\lambda\in K^\times$.  Since $K\subseteq Z(D)$,
\[
        f_b(\lambda x)(\lambda x)^{-1}=\lambda^{p^b-1}f_b(x)x^{-1}
\]
and
\[
        (\lambda x)^n g_a((\lambda x)^{-1})
        =\lambda^{n-p^a}x^n g_a(x^{-1}).
\]
Thus, for every $\lambda\in K^\times$,
\[
        \sum_{b\in B}\lambda^{p^b-1}f_b(x)x^{-1}
        +
        \sum_{a\in A}\lambda^{n-p^a}x^n g_a(x^{-1})=0.
\]
Lemma \ref{lem:laurent-independence-infinite} gives the coefficient identities.  The final assertion follows because all terms then occur in isolated exponent classes.
\end{proof}

\begin{remark}\label{rem:role-of-scaling}
The infinite-field support theorem complements the unconditional finite-central-field decomposition in Theorem \ref{thm:canonical-weight-pairing}.  Together, they isolate the central-weight obstruction responsible for the Frobenius phenomena without using generalized-polynomial-identity methods.
\end{remark}

\section{Characteristic two}\label{sec:char-two}

Throughout this section, $D$ is a division ring of characteristic two, $Z=Z(D)$, $n\geq2$, and
\[
        N=n+1.
\]
For $x,y\in D$, put $[x,y]=xy+yx$; in characteristic two this is the usual additive commutator.  We begin with the generalized-polynomial reduction and then use additivity more strongly to obtain vanishing theorems.

We use Hua's identity in the form
\begin{equation}\label{eq:hua}
        1-a=\left(1+(a^{-1}-1)^{-1}\right)^{-1},
\end{equation}
valid for $a\neq0,1$ in any division ring.

\begin{lemma}\label{lem:hua-reduction}
Let $f,g:D\to D$ be additive maps satisfying \eqref{eq:main}, and put
\[
        c=f(1)=g(1).
\]
Then, for every $a\in D$,
\begin{equation}\label{eq:fg-sum-formula}
        f(a)+g(a)=c+(1+a)^n c(1+a)+a^nca.
\end{equation}
\end{lemma}

\begin{proof}
Substituting $x=1$ into \eqref{eq:main} gives $f(1)+g(1)=0$, hence $f(1)=g(1)=c$.  Equation \eqref{eq:main} is equivalent to
\begin{equation}\label{eq:f-from-g-char2}
        f(x)=x^n g(x^{-1})x\qquad (x\in D^\times),
\end{equation}
and, after replacing $x$ by $x^{-1}$, to
\begin{equation}\label{eq:g-from-f-char2}
        g(x)=x^n f(x^{-1})x\qquad (x\in D^\times).
\end{equation}

Assume first that $a\neq0,1$.  Additivity, \eqref{eq:f-from-g-char2}, Hua's identity, and \eqref{eq:g-from-f-char2} give
\[
\begin{aligned}
        f(a)&=c+f(1-a),\\
        f(1-a)
        &=(1-a)^n g((1-a)^{-1})(1-a)\\
        &=(1-a)^n c(1-a)+a^n f(a^{-1}-1)a\\
        &=(1-a)^n c(1-a)+g(a)+a^nca.
\end{aligned}
\]
Since $1-a=1+a$ in characteristic two, this proves \eqref{eq:fg-sum-formula}.  The cases $a=0,1$ are immediate.
\end{proof}

For $c\in D$, define
\begin{equation}\label{eq:Pc-def}
        P_c(X)=c+(1+X)^n c(1+X)+X^n cX
\end{equation}
in $D*_{Z}Z\langle X\rangle$, and put
\begin{equation}\label{eq:Phi-def}
        \Phi_c(X,Y)=P_c(X+Y)+P_c(X)+P_c(Y).
\end{equation}
Lemma \ref{lem:hua-reduction} says that $f+g=P_c$ as functions on $D$.  Since $f+g$ is additive, $\Phi_c(a,b)=0$ for all $a,b\in D$.

\begin{lemma}\label{lem:gpi-polynomial-nonzero-corrected}
Let $0\neq c\in D$.  Then $\Phi_c$ is a nonzero generalized polynomial if either
\begin{enumerate}[label=\textup{(\roman*)}]
\item $c\notin Z$; or
\item $c\in Z$ and $N$ is not a power of two.
\end{enumerate}
If $c\in Z$ and $N$ is a power of two, then $P_c(X)=0$ and hence $\Phi_c(X,Y)=0$.
\end{lemma}

\begin{proof}
Suppose first that $c\notin Z$.  In free-product normal form, the word $X^{n-1}Yc$ occurs in $(1+X+Y)^n c(1+X+Y)$.  The polynomials $P_c(X)$ and $P_c(Y)$ contain no mixed word, while every word contributed by $(X+Y)^n c(X+Y)$ has a variable to the right of $c$.  Thus the displayed word cannot cancel and $\Phi_c\neq0$.

Now suppose $c\in Z$.  Then
\[
        P_c(X)=c\bigl(1+(1+X)^N+X^N\bigr).
\]
If $N=2^r$, the Frobenius identity $(1+X)^{2^r}=1+X^{2^r}$ in the one-variable polynomial algebra gives $P_c=0$.

If $N$ is not a power of two, Lucas' theorem gives an integer $d$ with $2\leq d\leq N-1$ for which $\binom Nd$ is odd.  The homogeneous degree-$d$ part of $\Phi_c$ is
\[
        c\binom Nd\bigl((X+Y)^d+X^d+Y^d\bigr),
\]
which contains the word $cX^{d-1}Y$ with nonzero coefficient.  Hence $\Phi_c\neq0$.
\end{proof}

\begin{theorem}\label{thm:char-two-gpi-reduction}
Let $f,g:D\to D$ be additive maps satisfying \eqref{eq:main}, and set $c=f(1)=g(1)$.  If $c\neq0$ and either $c\notin Z$ or $N$ is not a power of two, then $D$ is finite-dimensional over $Z$.

If $c\in Z$ and $N$ is a power of two, then $f=g$.  Consequently, if $[D:Z]=\infty$, every solution has $f=g=h$, where
\begin{equation}\label{eq:one-map-reduced}
        h(x)x^{-1}+x^n h(x^{-1})=0\qquad (x\in D^\times),
\end{equation}
and
\[
        h(1)\in Z,
        \qquad
        h(1)=0\quad\text{or}\quad N\text{ is a power of two}.
\]
\end{theorem}

\begin{proof}
By Lemma \ref{lem:hua-reduction}, additivity of $f+g$ gives the generalized polynomial identity $\Phi_c(X,Y)=0$.  Under the first set of hypotheses, Lemma \ref{lem:gpi-polynomial-nonzero-corrected} shows that this GPI is nontrivial, so Theorem \ref{thm:martindale} gives $[D:Z]<\infty$.

If $c\in Z$ and $N$ is a power of two, then $P_c=0$, and Lemma \ref{lem:hua-reduction} gives $f+g=0$, or $f=g$.  Finally suppose $[D:Z]=\infty$.  The first part rules out every nonzero $c$ except the central power-of-two case.  If $c=0$, then $P_c=0$ as well.  Thus $f=g=h$ in every case, and the remaining assertions follow.
\end{proof}

\begin{remark}\label{rem:exception-real}
When $c\in Z$ and $N$ is a power of two, the GPI vanishes identically.  This does not obstruct the reduction to one map: it already implies $f+g=0$.  It does, however, explain why a different argument is needed to decide whether $c$ or $h$ must vanish.
\end{remark}

\subsection{An inverse-free identity}

The fractional-linear transformations $x\mapsto x+1$ and $x\mapsto x^{-1}$ generate only a six-point orbit.  Additivity supplies the missing relation: for $x\neq0,1$,
\[
        x^{-1}+(x+1)^{-1}=(x^2+x)^{-1}.
\]
Applying the reduced identity to this sum removes all inverses.

\begin{lemma}\label{lem:inverse-free}
Let $h:D\to D$ be additive and satisfy \eqref{eq:one-map-reduced}.  Then, for every $x\in D$,
\begin{equation}\label{eq:inverse-free}
\begin{aligned}
        h(x^2)={}&(x+1)^n h(x)(x+1)
        +x^n h(x+1)x+h(x).
\end{aligned}
\end{equation}
Equivalently, if $c=h(1)$, then
\begin{equation}\label{eq:inverse-free-c}
\begin{aligned}
        h(x^2)={}&(x+1)^n h(x)(x+1)
        +x^n h(x)x+h(x)+x^n c x.
\end{aligned}
\end{equation}
\end{lemma}

\begin{proof}
The cases $x=0,1$ are immediate.  Let $x\neq0,1$, set $A=h(x)$ and $w=x^2+x=x(x+1)=(x+1)x$.  The elements $x$, $x+1$, and $w$ commute with one another and with their inverses.  From \eqref{eq:one-map-reduced},
\[
        h(x^{-1})=x^{-n}Ax^{-1},
        \qquad
        h((x+1)^{-1})=(x+1)^{-n}h(x+1)(x+1)^{-1}.
\]
Since $w^{-1}=x^{-1}+(x+1)^{-1}$, additivity and the reduced identity at $w$ give
\[
\begin{aligned}
        h(w)
        &=w^n h(w^{-1})w\\
        &=(x+1)^nA(x+1)+x^n h(x+1)x.
\end{aligned}
\]
Finally, $h(w)=h(x^2)+h(x)$.  This proves \eqref{eq:inverse-free}, and \eqref{eq:inverse-free-c} follows from $h(x+1)=h(x)+c$.
\end{proof}

\begin{remark}\label{rem:inverse-free-specializations}
Two specializations will be useful.  If $c\in Z$ and $n=2^s$, then \eqref{eq:inverse-free-c} reduces to
\begin{equation}\label{eq:inverse-free-power-n}
        h(x^2)=x^n h(x)+h(x)x+cx^N.
\end{equation}
If $c\in Z$ and $N=2^r$, so $n=2^r-1$, then every binomial coefficient $\binom nk$ is odd, and \eqref{eq:inverse-free-c} becomes
\begin{equation}\label{eq:inverse-free-power-N}
\begin{aligned}
        h(x^2)={}&x^n h(x)+h(x)x
        +\sum_{k=1}^{n-1}\bigl(x^kh(x)x+x^kh(x)\bigr)+cx^N.
\end{aligned}
\end{equation}
\end{remark}

\subsection{The noncommutative case for \texorpdfstring{$n=2$}{n=2}}

We first record the reduction available without any hypothesis on the center or central dimension.

\begin{proposition}\label{prop:n2-reduction}
Let $D$ be noncommutative, and suppose additive maps $f,g:D\to D$ satisfy
\[
        f(x)x^{-1}+x^2g(x^{-1})=0\qquad (x\in D^\times).
\]
Then $f=g=h$ and $h(1)=0$.
\end{proposition}

\begin{proof}
Let $c=f(1)=g(1)$.  Lemma \ref{lem:hua-reduction} with $n=2$ gives
\[
        f(a)+g(a)=c+(1+a)^2c(1+a)+a^2ca=ca+a^2c.
\]
Since $f+g$ is additive, linearizing the last expression gives $(ab+ba)c=0$ for all $a,b\in D$.  If $c\neq0$, then $D$ is commutative, a contradiction.  Hence $c=0$, and Lemma \ref{lem:hua-reduction} yields $f+g=0$, so $f=g=h$.
\end{proof}

By \eqref{eq:inverse-free-power-n}, the reduced map in Proposition \ref{prop:n2-reduction} satisfies
\begin{equation}\label{eq:n2-square}
        h(x^2)=x^2h(x)+h(x)x\qquad (x\in D).
\end{equation}

\begin{proposition}\label{prop:n2-cocycle}
Under the hypotheses of Proposition \ref{prop:n2-reduction},
\begin{equation}\label{eq:n2-cocycle}
        [x,y]h(z)+[y,z]h(x)+[z,x]h(y)=0
        \qquad (x,y,z\in D).
\end{equation}
\end{proposition}

\begin{proof}
Replace $x$ by $x+y$ in \eqref{eq:n2-square}.  Since $(x+y)^2=x^2+y^2+[x,y]$, subtracting the instances of \eqref{eq:n2-square} at $x$ and $y$ gives
\begin{equation}\label{eq:n2-linearized-square}
\begin{aligned}
        h([x,y])={}&x^2h(y)+y^2h(x)+[x,y]h(x)+[x,y]h(y)\\
        &+h(x)y+h(y)x.
\end{aligned}
\end{equation}
Fix $y$, replace $x$ by $x+z$, and take the second difference in $x$.  The left-hand side and all terms on the right except $x^2h(y)$ and $[x,y]h(x)$ are additive in $x$.  Their second differences are, respectively,
\[
        [x,z]h(y)
        \quad\text{and}\quad
        [x,y]h(z)+[z,y]h(x).
\]
Their sum is zero, which is \eqref{eq:n2-cocycle} after relabeling.
\end{proof}

\begin{theorem}\label{thm:n2-char-two}
Let $D$ be a noncommutative division ring of characteristic two.  If additive maps $f,g:D\to D$ satisfy
\[
        f(x)x^{-1}+x^2g(x^{-1})=0\qquad (x\in D^\times),
\]
then $f=g=0$.
\end{theorem}

\begin{proof}
By Proposition \ref{prop:n2-reduction}, $f=g=h$ and $h(1)=0$.  Suppose $h\neq0$.

First, $h$ vanishes on $Z$.  Indeed, putting $z=\lambda\in Z$ in \eqref{eq:n2-cocycle} gives $[x,y]h(\lambda)=0$.  Some commutator $[x,y]$ is nonzero, so $h(\lambda)=0$.

We next show that $Z=\F_2$.  Put $y=\lambda\in Z^\times$ in \eqref{eq:n2-linearized-square}.  Using $h(\lambda)=0$ gives
\[
        (\lambda^2+\lambda)h(x)=0.
\]
If $\lambda\notin\F_2$, the central coefficient is nonzero, forcing $h=0$.  Thus no such $\lambda$ exists.

Choose $b\notin\ker h$ and put $c=h(b)\neq0$.  We first record a consequence that will also be applied to $b+u$ below.  Let $d\in D$ satisfy $h(d)\neq0$, and put $e=h(d)$.  Since $h$ vanishes on $Z$, we have $d\notin Z$.  If $x,z\in C_D(d)$, then \eqref{eq:n2-cocycle} applied to $(x,z,d)$ gives $[x,z]e=0$; hence $C_D(d)$ is commutative.  If merely $x\in C_D(d)$ and $z\in D$, the same identity gives
\[
        [x,z]e=[z,d]h(x),
\]
or
\begin{equation}\label{eq:n2-elementary-operator}
        xze+z\bigl(xe+dh(x)\bigr)+dzh(x)=0
        \qquad(z\in D).
\end{equation}
If $1,x,d$ were linearly independent over $Z$, Lemma \ref{lem:elementary-operators} applied to \eqref{eq:n2-elementary-operator} would give $e=0$.  Therefore $1,x,d$ are dependent.  Since $d\notin Z$, this implies $x\in Z+Zd$.  Consequently, for every $d\notin\ker h$,
\[
        C_D(d)=Z+Zd=\F_2+\F_2d\cong\F_4,
\]
and hence $d^2=d+1$.  In particular, $b^2=b+1$.

We now prove $\ker h=\F_2$.  Equation \eqref{eq:n2-square} shows that $u\in\ker h$ implies $u^2\in\ker h$.  Let $0\neq u\in\ker h$.  Then $h(b+u)=c\neq0$, so the general conclusion just proved, applied to $d=b+u$, gives $(b+u)^2=b+u+1$.  Using $b^2=b+1$, we obtain
\begin{equation}\label{eq:n2-kernel-commutator}
        [b,u]=u^2+u.
\end{equation}
Applying the same identity to $u^2\in\ker h$ gives $[b,u^2]=u^4+u^2$.  On the other hand, the inner derivation $[b,\cdot]$ satisfies
\[
        [b,u^2]=[b,u]u+u[b,u]=0
\]
by \eqref{eq:n2-kernel-commutator}.  Thus $u^4=u^2$, and the division-ring property gives $u^2=1$.  In characteristic two, $(u+1)^2=0$, so $u=1$.

It follows that every $x\in D\setminus\F_2$ satisfies $x^2=x+1$.  The ring $D$ is infinite by Wedderburn's little theorem.  Choose
\[
        x\notin\F_2,
        \qquad
        y\notin\{0,1,x,x+1\}.
\]
Then $x,y,x+y\notin\F_2$, and comparison of $(x+y)^2=(x+y)+1$ with $x^2=x+1$ and $y^2=y+1$ gives $[x,y]=1$.  Hence
\[
        [x,y^2]=[x,y]y+y[x,y]=0,
\]
whereas $y^2=y+1$ gives $[x,y^2]=[x,y]=1$, a contradiction.  Therefore $h=0$.
\end{proof}

\begin{remark}\label{rem:n2-scope}
Noncommutativity is essential: Example \ref{ex:F4} contains nonzero solutions over $\F_4$, including reduced solutions $h(x)=a(x+x^2)$ with $h(1)=0$.  Thus Theorem \ref{thm:n2-char-two} completes the $n=2$ case for noncommutative division rings, while Section~\ref{sec:finite-fields} classifies the finite commutative case.  For arbitrary additive maps on infinite fields, see the established field-equation results discussed in Remark~\ref{rem:infinite-field-additive}.
\end{remark}

\subsection{Automatic semilinearity and infinite centers}

\begin{lemma}\label{lem:central-square-semilinearity}
Assume $[D:Z]=\infty$, and let $h$ satisfy \eqref{eq:one-map-reduced}.  Then
\begin{equation}\label{eq:central-square-semilinearity}
        h(\lambda^2x)=\lambda^N h(x)
        \qquad(\lambda\in Z^\times,\ x\in D).
\end{equation}
\end{lemma}

\begin{proof}
Fix $\lambda\in Z^\times$ and define
\[
        f_\lambda(x)=h(\lambda x),
        \qquad
        g_\lambda(x)=\lambda^N h(\lambda^{-1}x).
\]
Evaluating \eqref{eq:one-map-reduced} at $\lambda x$ and multiplying by $\lambda$ shows that $(f_\lambda,g_\lambda)$ satisfies \eqref{eq:main}.  Since $[D:Z]=\infty$, Theorem \ref{thm:char-two-gpi-reduction} gives $f_\lambda=g_\lambda$.  Thus
\[
        h(\lambda x)=\lambda^N h(\lambda^{-1}x).
\]
Replacing $x$ by $\lambda x$ proves \eqref{eq:central-square-semilinearity}.
\end{proof}

\begin{theorem}\label{thm:char-two-infinite-center-nonpower}
Suppose $[D:Z]=\infty$, the center $Z$ is infinite, and $N$ is not a power of two.  Then every additive solution of \eqref{eq:main} is zero.
\end{theorem}

\begin{proof}
By Theorem \ref{thm:char-two-gpi-reduction}, $f=g=h$.  For $\lambda,\mu\in Z^\times$ with $\lambda+\mu\neq0$, additivity and Lemma \ref{lem:central-square-semilinearity} give
\[
\begin{aligned}
        (\lambda+\mu)^N h(x)
        &=h((\lambda+\mu)^2x)\\
        &=h(\lambda^2x)+h(\mu^2x)
        =(\lambda^N+\mu^N)h(x).
\end{aligned}
\]
Because $N$ is not a power of two, Lucas' theorem shows that
\[
        P(X,Y)=(X+Y)^N+X^N+Y^N
\]
is a nonzero polynomial over $\F_2$.  Since $Z$ is infinite, the nonzero polynomial $XY(X+Y)P(X,Y)$ has a non-root $(\lambda,\mu)\in Z^2$.  For this pair the preceding identity forces $h(x)=0$ for every $x$.
\end{proof}

\begin{theorem}\label{thm:char-two-infinite-center-power}
Suppose $[D:Z]=\infty$, the center $Z$ is infinite, and $N=2^r$ with $r\geq2$.  Then every additive solution of \eqref{eq:main} is zero.
\end{theorem}

\begin{proof}
By Theorem \ref{thm:char-two-gpi-reduction}, $f=g=h$ and $c=h(1)\in Z$.  Put $A=h(x)$.  We proceed in three steps.

First, $A$ commutes with $x$.  Apply \eqref{eq:inverse-free-power-N} at $X=\lambda^2x$, where $\lambda\in Z^\times$.  Lemma \ref{lem:central-square-semilinearity} gives
\[
        h(X)=\lambda^N A,
        \qquad
        h(X^2)=\lambda^{2N}h(x^2).
\]
Hence
\begin{equation}\label{eq:power-case-scaled}
\begin{aligned}
\lambda^{2N}h(x^2)={}&\lambda^{3N-2}x^nA+\lambda^{N+2}Ax\\
&+\sum_{k=1}^{n-1}\left(
\lambda^{N+2k+2}x^kAx+\lambda^{N+2k}x^kA\right)
+\lambda^{2N}cx^N.
\end{aligned}
\end{equation}
Compare this with $\lambda^{2N}$ times \eqref{eq:inverse-free-power-N}.  By Lemma \ref{lem:laurent-independence-infinite}, each coefficient of a power of $\lambda$ vanishes.  Since $N\geq4$, the coefficient of $\lambda^{N+2}$ comes only from $Ax$ and from the term $x^kA$ with $k=1$.  Therefore
\[
        Ax+xA=0,
        \qquad\text{or equivalently}\qquad
        h(x)x=xh(x).
\]

Second, $h$ is central-valued.  Linearizing $[h(x),x]=0$ gives
\[
        [h(x),y]=[h(y),x]
        \qquad(x,y\in D).
\]
Using \eqref{eq:central-square-semilinearity}, for every $\lambda\in Z^\times$ we obtain
\[
\begin{aligned}
        \lambda^N[h(x),y]
        &=[h(\lambda^2x),y]
         =[h(y),\lambda^2x]
         =\lambda^2[h(x),y].
\end{aligned}
\]
Because $Z$ is infinite and $N\geq4$, choose $\lambda$ with $\lambda^{N-2}\neq1$.  It follows that $[h(x),y]=0$ for all $x,y$, so $h(D)\subseteq Z$.

Finally, centrality of the values in \eqref{eq:one-map-reduced} gives
\[
        h(x)=h(x^{-1})x^N.
\]
Thus $h(x)\neq0$ implies $x^N\in Z$.  If $x^N\in Z$ for every $x\in D$, then every element of $D$ is algebraic over $Z$ of degree at most $N$, contrary to Theorem \ref{thm:bounded-degree} and $[D:Z]=\infty$.  Choose $x_0\in D$ with $x_0^N\notin Z$; then $h(x_0)=0$.

If $h(y)\neq0$ for some $y\in D$, then for every $\nu\in Z^\times$,
\[
        h(y+\nu^2x_0)=h(y)\neq0,
\]
and hence $(y+\nu^2x_0)^N\in Z$.  Since $\nu^2$ is central, write
\[
        (y+\nu^2x_0)^N=\sum_{k=0}^N(\nu^2)^kW_k,
\]
where $W_k$ is the sum of all words with $k$ letters $x_0$ and $N-k$ letters $y$.  The set of squares in the infinite field $Z$ is infinite.  Applying Lemma \ref{lem:laurent-independence-infinite} in the $Z$-vector space $D/Z$ gives $W_k\in Z$ for every $k$.  In particular, $W_N=x_0^N\in Z$, a contradiction.  Hence $h=0$.
\end{proof}

\begin{corollary}\label{cor:char-two-infinite-center}
Let $D$ be a division ring of characteristic two with infinite center and $[D:Z(D)]=\infty$.  Then, for every $n\geq2$, the identity \eqref{eq:main} forces $f=g=0$.
\end{corollary}

\begin{proof}
Here $N=n+1\geq3$.  If $N$ is not a power of two, apply Theorem \ref{thm:char-two-infinite-center-nonpower}.  If it is a power of two, then $N\geq4$, and Theorem \ref{thm:char-two-infinite-center-power} applies.
\end{proof}

\section{Finite centers in characteristic two}\label{sec:finite-center-char-two}

Throughout this section $D$ is a division ring with $\charac D = 2$, $Z = Z(D)$, $n \ge 2$ and
$N = n+1$. As in Section~\ref{sec:char-two}, we write $[x,y] = xy+yx$, and all binomial
coefficients are read modulo $2$.

\subsection{A compact form of the inverse-free identity}

The inverse-free identity \eqref{eq:inverse-free-c} becomes more useful after the two occurrences
of $(x+1)^n$ are expanded and the resulting sums are reindexed. In the resulting normal form,
each bracket is homogeneous of a single degree in $x$ and is therefore compatible with the
central weight decomposition.

\begin{lemma}\label{L:compact}
Let $h : D \to D$ be additive and satisfy \eqref{eq:one-map-reduced}, and put $c = h(1)$. Then,
for every $x \in D$,
\begin{equation}\label{E:compact}
\begin{aligned}
h(x^2)&=\sum_{j=1}^{n} A_j(x)+x^n c x,\\
A_j(x)&=\binom{n}{j}x^j h(x)
       +\binom{n}{j-1}x^{j-1}h(x)x.
\end{aligned}
\end{equation}
If $c \in Z$, the last term is $c\,x^{N}$.
\end{lemma}

\begin{proof}
Write $b_k = \binom{n}{k} \bmod 2$. Then
\[
\begin{aligned}
(x+1)^n&=\sum_{k=0}^{n} b_k x^k,\\
(x+1)^n h(x)(x+1)
&=\sum_{k=0}^{n} b_k x^k h(x)x+\sum_{k=0}^{n} b_k x^k h(x).
\end{aligned}
\]
Substituting into \eqref{eq:inverse-free-c},
\[
h(x^2) = \sum_{k=0}^{n} b_k x^k h(x)\,x \;+\; \sum_{k=0}^{n} b_k x^k h(x) \;+\; x^n h(x)\,x \;+\; h(x) \;+\; x^n c\,x .
\]
Since $b_n = b_0 = 1$, the term $k=n$ of the first sum cancels the isolated $x^n h(x)x$ and the
term $k=0$ of the second sum cancels the isolated $h(x)$, leaving
\[
h(x^2) = \sum_{k=0}^{n-1} b_k\, x^k h(x)\, x \;+\; \sum_{k=1}^{n} b_k\, x^k h(x) \;+\; x^n c\, x .
\]
Reindexing the first sum by $j = k+1$ gives \eqref{E:compact}.
\end{proof}

\begin{remark}
For $n = 2^s$ only the brackets $j=1$ and $j=n$ survive, and \eqref{E:compact} reduces to
\eqref{eq:inverse-free-power-n}. If $N = 2^r$, all $\binom{n}{j}$ are odd, and
\eqref{eq:inverse-free-power-N} is recovered. Evaluating \eqref{E:compact} at $x=1$ gives
$c = (2^n-1)c + (2^n-1)c + c$, as expected.
\end{remark}

Pascal's rule gives the following consequence: the two terms in $A_j$ combine whenever the values
of $h$ commute with their arguments.

\begin{corollary}\label{C:pascal}
If $h$ satisfies \eqref{eq:one-map-reduced}, $c = h(1) \in Z$, and $[h(x),x] = 0$ for all
$x \in D$, then
\[
h(x^2) \;=\; \bigl[(x+1)^{N} + x^{N} + 1\bigr]\, h(x) \;+\; c\,x^{N} \qquad (x \in D).
\]
In particular $h(x^2) = c\,x^{N}$ whenever $N$ is a power of two.
\end{corollary}

\begin{proof}
Commutation and Pascal's rule turn the $j$-th bracket into
\[
\left(\binom{n}{j}+\binom{n}{j-1}\right)x^jh(x)=\binom{N}{j}x^jh(x).
\]
Summing over $j$ gives
\[
\sum_{j=1}^{n}\binom{N}{j}x^{j}=(x+1)^{N}+x^{N}+1.
\]
\end{proof}

\subsection{A Frobenius root lemma}

The next lemma will be used to exploit relations of the form $x^N\in Z$. When $N$ is a power of
two and $Z$ is perfect, such a relation already forces $x$ itself to be central.

\begin{lemma}\label{L:frobroot}
Let $D$ be a division ring of characteristic two whose center $Z$ is perfect, and let $r \ge 0$.
If $x \in D$ satisfies $x^{2^{r}} \in Z$, then $x \in Z$.
\end{lemma}

\begin{proof}
Put $\alpha = x^{2^{r}} \in Z$. Since $Z$ is perfect there is $\beta \in Z$ with
$\beta^{2^{r}} = \alpha$. As $\beta$ is central, $(x+\beta)^2 = x^2 + \beta^2$, and iterating,
$(x+\beta)^{2^{r}} = x^{2^{r}} + \beta^{2^{r}} = 0$. A division ring has no nonzero nilpotent
elements, so $x = \beta \in Z$.
\end{proof}

\begin{corollary}\label{C:oddpart}
Let $Z$ be perfect, let $M \ge 1$, and write $M = 2^{a}M_0$ with $M_0$ odd. If $x^{M} \in Z$
then $x^{M_0} \in Z$.
\end{corollary}

\begin{proof}
Apply Lemma \ref{L:frobroot} to $u = x^{M_0}$, noting $u^{2^{a}} = x^{M} \in Z$.
\end{proof}

We shall combine the Frobenius root lemma with the following elementary observation.

\begin{lemma}\label{L:additivity}
Let $D \ne Z$ and let $h : D \to D$ be additive. If $h(x) \ne 0$ implies $x \in Z$, then
$h = 0$.
\end{lemma}

\begin{proof}
Suppose $h \ne 0$ and choose $\lambda$ with $h(\lambda) \ne 0$; by hypothesis $\lambda \in Z$.
Pick $x \in D \setminus Z$. Then $x \notin Z$ and $\lambda + x \notin Z$, so
$h(x) = h(\lambda+x) = 0$ and hence $h(\lambda) = h(\lambda+x) + h(x) = 0$, a contradiction.
\end{proof}

\begin{remark}\label{R:shortcut613}
Lemmas \ref{L:frobroot} and \ref{L:additivity} shorten the third step of
Theorem~\ref{thm:char-two-infinite-center-power} whenever $Z$ is perfect. There $N = 2^{r}$,
$h$ has already been shown to be central-valued, and
$h(x) = h(x^{-1})x^{N}$ gives $h(x) \ne 0 \Rightarrow x^{N} \in Z$; Lemma \ref{L:frobroot} then
gives $x \in Z$ and Lemma \ref{L:additivity} gives $h = 0$ at once, with no appeal to
Theorem~\ref{thm:bounded-degree} and no expansion of $(y+\nu^2 x_0)^N$. The expansion argument in
Theorem~\ref{thm:char-two-infinite-center-power} remains necessary for imperfect infinite centers
such as $\F_2(t)$.
\end{remark}

\subsection{The graded system over a finite center}

We now turn to finite centers and impose the following standing hypotheses:
\begin{equation}\label{E:standing}
\charac D = 2, \qquad Z = Z(D) = \Fq, \qquad q = 2^{m} \ge 4, \qquad [D:Z] = \infty .
\end{equation}
By Theorem~\ref{thm:char-two-gpi-reduction}, every solution of \eqref{eq:main} satisfies
$f = g = h$, where $h$ is additive, $c = h(1) \in Z$, and \eqref{eq:one-map-reduced} holds.
By Lemma~\ref{lem:central-square-semilinearity},
$h(\lambda^{2}x) = \lambda^{N}h(x)$ for $\lambda \in Z^{\times}$;
since squaring is an automorphism of $\Fq$, this says precisely that $h$ is homogeneous of a single
weight,
\begin{equation}\label{E:weight}
\begin{aligned}
h(\lambda x)&=\lambda^s h(x)
   &&(\lambda\in\Fq^\times,\ x\in D),\\
s&\equiv N\,2^{m-1}\pmod{q-1},
   &2s&\equiv N\pmod{q-1}.
\end{aligned}
\end{equation}

The point of the next lemma is that the single identity \eqref{E:compact} is only the $\lambda = 1$
shadow of a system of $q-1$ identities, one for each residue class of degrees.

\begin{lemma}\label{L:graded}
Assume \eqref{E:standing}. For every $t \in \mathbb{Z}/(q-1)\mathbb{Z}$ and every $x \in D$,
\begin{equation}\label{E:graded}
\sum_{\substack{1 \le j \le n \\ j \equiv t \!\!\pmod{q-1}}} A_j(x) \;=\;
\begin{cases}
h(x^{2}) + c\,x^{N}, & t \equiv s \pmod{q-1},\\[2pt]
0, & \text{otherwise.}
\end{cases}
\end{equation}
\end{lemma}

\begin{proof}
Fix $x \in D$ and let $\lambda \in \Fq^{\times}$. Evaluate \eqref{E:compact} at $\lambda x$. On the
left, $h((\lambda x)^{2}) = h(\lambda^{2}x^{2}) = \lambda^{N}h(x^{2})$ by
Lemma~\ref{lem:central-square-semilinearity}. On the right,
centrality of $\lambda$ together with \eqref{E:weight} gives
\[
\binom{n}{j}(\lambda x)^{j}h(\lambda x) + \binom{n}{j-1}(\lambda x)^{j-1}h(\lambda x)(\lambda x)
= \lambda^{s+j}A_j(x),
\]
while the last term becomes $\lambda^{N}c\,x^{N}$. Hence
\[
\lambda^{N}\bigl(h(x^{2}) + c\,x^{N}\bigr) = \sum_{j=1}^{n}\lambda^{s+j}A_j(x)
\qquad (\lambda \in \Fq^{\times}).
\]
Dividing by $\lambda^{s}$ and using $N - s \equiv s \pmod{q-1}$ from \eqref{E:weight},
\[
\lambda^{s}\bigl(h(x^{2}) + c\,x^{N}\bigr) = \sum_{j=1}^{n}\lambda^{j}A_j(x)
\qquad (\lambda \in \Fq^{\times}).
\]
Multiply by $\lambda^{-t}$, sum over $\lambda \in \Fq^{\times}$, and use the
character-orthogonality relation from Lemma~\ref{lem:canonical-weight-decomposition}, which in
characteristic two reads
$\sum_{\lambda \in \Fq^{\times}} \lambda^{u} = 1$ if $(q-1) \mid u$ and $0$ otherwise. This gives
\eqref{E:graded}.
\end{proof}

Everything now hinges on the bracket $A_1$, which is isolated in its degree class exactly when
$n \le q-1$.

\begin{lemma}\label{L:A1}
Assume \eqref{E:standing} and $2 \le n \le q-1$. Then $s \not\equiv 1 \pmod{q-1}$ and
\begin{equation}\label{E:A1}
n\,x\,h(x) + h(x)\,x = 0 \qquad (x \in D).
\end{equation}
\end{lemma}

\begin{proof}
If $s \equiv 1$, then $N \equiv 2s \equiv 2$ by \eqref{E:weight}, i.e. $n \equiv 1 \pmod{q-1}$;
but $2 \le n \le q-1$ excludes both $n=1$ and $n=q$, so this is impossible. Since $n \le q-1$, the
integers $1,2,\dots,n$ are pairwise incongruent modulo $q-1$, so the class $t \equiv 1$ meets
$\{1,\dots,n\}$ only at $j = 1$. As $s \not\equiv 1$, \eqref{E:graded} with $t=1$ gives
$A_1(x) = \binom{n}{1}xh(x) + \binom{n}{0}h(x)x = 0$.
\end{proof}

\subsection{Vanishing below the order of the center}

\begin{theorem}\label{T:even}
Assume \eqref{E:standing} and let $n$ be even with $2 \le n \le q-1$. Then $f = g = 0$.
\end{theorem}

\begin{proof}
By Lemma \ref{L:A1}, $h(x)x = n\,xh(x) = 0$ for every $x$, so $h(x) = 0$ for every $x \ne 0$; and
$h(0) = 0$.
\end{proof}

\begin{proposition}\label{P:central}
Assume \eqref{E:standing} and let $n$ be odd with $3 \le n \le q-1$. Then
\[
[h(x),x] = 0 \quad (x \in D), \qquad h(D) \subseteq Z, \qquad h(x) = h(x^{-1})\,x^{N} \quad (x \in D^{\times}).
\]
In particular $h(x) \ne 0$ implies $x^{N} \in Z$.
\end{proposition}

\begin{proof}
For $n$ odd, \eqref{E:A1} reads $[h(x),x] = 0$. Replacing $x$ by $x+y$ and subtracting the
instances at $x$ and $y$ gives the linearized form $[h(x),y] = [h(y),x]$ for all $x,y \in D$. Fix
$x,y$ and let $\lambda \in \Fq^{\times}$. Applying the linearized form to the pair $(\lambda x, y)$
and using \eqref{E:weight} and the centrality of $\lambda$,
\[
\lambda^{s}[h(x),y] = [h(\lambda x),y] = [h(y),\lambda x] = \lambda\,[h(y),x] = \lambda\,[h(x),y],
\]
so $(\lambda^{s} + \lambda)[h(x),y] = 0$. By Lemma \ref{L:A1} we have $s \not\equiv 1 \pmod{q-1}$,
so $\lambda \mapsto \lambda^{s-1}$ is not identically $1$ on $\Fq^{\times}$; choosing $\lambda$
with $\lambda^{s-1} \ne 1$ gives $[h(x),y] = 0$ for all $x,y$, that is, $h(D) \subseteq Z$.
Finally, \eqref{eq:one-map-reduced} reads $h(x)x^{-1} = x^{n}h(x^{-1})$, and
$h(x^{-1}) \in Z$ may be moved past
$x^{n}$, giving $h(x) = h(x^{-1})x^{N}$. If $h(x) \ne 0$ then $h(x^{-1}) \ne 0$ and
$x^{N} = h(x^{-1})^{-1}h(x) \in Z$.
\end{proof}

\begin{theorem}\label{T:odd}
Assume \eqref{E:standing} and let $n$ be odd with $3 \le n \le q-1$. Then $f = g = 0$.
\end{theorem}

\begin{proof}
Suppose $h \ne 0$ and set $K = \ker h$, an $\Fq$-subspace of $D$ by \eqref{E:weight}. Write
$N = 2^{a}N_0$ with $N_0$ odd; since $n$ is odd, $N$ is even, so $a \ge 1$ and
\[
N_0 \le N/2 \le q/2 = 2^{m-1} \le 2^{m}-2 = q-2 .
\]
By Proposition \ref{P:central} and Corollary \ref{C:oddpart},
\begin{equation}\label{E:offK}
x \notin K \;\Longrightarrow\; x^{N_0} \in Z .
\end{equation}
We claim the same holds for $x \in K$. Fix $y \notin K$, let $k \in K$, and let
$\mu \in \Fq^{\times}$. Then $h(y + \mu k) = h(y) + \mu^{s}h(k) = h(y) \ne 0$, so
$(y+\mu k)^{N_0} \in Z$ by \eqref{E:offK}. As $\mu$ is central,
\[
(y+\mu k)^{N_0} \;=\; \sum_{i=0}^{N_0} \mu^{i}\,U_i, \qquad
U_i \;=\; \sum_{\substack{w \text{ a word of length } N_0 \text{ in } y,k \\ \text{with exactly } i \text{ letters } k}} w ,
\]
with $U_i$ independent of $\mu$. Passing to the $\Fq$-vector space $D/Z$ we obtain
$\sum_{i=0}^{N_0}\mu^{i}\overline{U_i} = 0$ for every $\mu \in \Fq^{\times}$. Since
$0 \le i \le N_0 \le q-2$, these exponents are pairwise incongruent modulo $q-1$, so the
orthogonality relation used in Lemma \ref{L:graded} forces $\overline{U_i} = 0$ for every $i$; in
particular $U_{N_0} = k^{N_0} \in Z$.

Hence $x^{N_0} \in Z$ for every $x \in D$, so every element of $D$ is algebraic over $Z$ of
degree at most $N_0$. Theorem~\ref{thm:bounded-degree} gives $[D:Z] < \infty$, contradicting
\eqref{E:standing}.
Therefore $h = 0$.
\end{proof}

\begin{remark}
When $N_0 = 1$, equivalently when $N$ is a power of two, \eqref{E:offK} gives
$x \notin K \Rightarrow x \in Z$, and Lemma~\ref{L:additivity} finishes the proof directly.
This includes the endpoint $n = q-1$, where $N = q$.
\end{remark}

\begin{theorem}\label{T:main}
Let $D$ be a division ring of characteristic two with finite center $Z(D) = \Fq$ and
$[D:Z(D)] = \infty$. Let $2 \le n \le q-1$ and let $f,g : D \to D$ be additive maps satisfying
\eqref{eq:main}. Then $f = g = 0$.
\end{theorem}

\begin{proof}
The hypothesis forces $q \ge 4$, since for $q = 2$ the range is empty. By
Theorem~\ref{thm:char-two-gpi-reduction}, $f = g = h$, with $h$ as in \eqref{E:standing}.
Apply Theorem~\ref{T:even} if $n$ is even and Theorem~\ref{T:odd} if $n$ is odd.
\end{proof}

\begin{corollary}\label{C:summary}
Let $D$ be a division ring of characteristic two with $[D:Z(D)] = \infty$, and let $n \ge 2$.
If $Z(D)$ is infinite, then \eqref{eq:main} forces $f = g = 0$ for every $n$. If $Z(D)$ is
finite, then \eqref{eq:main} forces $f = g = 0$ for every $n < |Z(D)|$.
\end{corollary}

\begin{proof}
The infinite-center case is Corollary~\ref{cor:char-two-infinite-center}; the finite-center case
is Theorem~\ref{T:main}.
\end{proof}

\subsection{Remaining cases}

For division rings infinite-dimensional over their centers, the remaining characteristic-two cases have finite center: $Z(D)=\F_q$ with $q\geq4$ and $n\geq q$, or $Z(D)=\F_2$ and $n\geq3$.  The case of centrally finite, noncommutative division rings with $n\geq3$ also remains outside the vanishing results proved here.  The first limitation enters at Lemma~\ref{L:A1}. If $n\geq q$, the residue class of
$1$ modulo $q-1$ meets $\{1,\dots,n\}$ in $1,q,2q-1,\dots$, so \eqref{E:graded} yields the weaker
relation
\[
\sum_{\substack{k \ge 0\\1+k(q-1) \le n}}
\left[
\binom{n}{1+k(q-1)}x^{1+k(q-1)}h(x)
+\binom{n}{k(q-1)}x^{k(q-1)}h(x)x
\right]=0,
\]
and the powers $x^{k(q-1)}$ do not collapse in $D$. If $Z=\F_2$, the grading is vacuous.

\begin{problem}\label{prob:finite-center-large-n}
Let $D$ be a division ring of characteristic two with $Z(D)=\Fq$, where $q\geq4$, and
$[D:Z(D)]=\infty$. For $n\geq q$, can \eqref{eq:one-map-reduced} admit a nonzero additive
solution?
\end{problem}

\begin{problem}\label{prob:center-F2}
Let $D$ be a noncommutative division ring with $Z(D)=\F_2$ and
$[D:Z(D)]=\infty$. For $n\geq3$, can \eqref{eq:one-map-reduced} admit a nonzero additive
solution?
\end{problem}

The finite-field examples do not automatically extend to these division rings: they use the
additivity of the Frobenius maps $x\mapsto x^{p^r}$, whereas mixed noncommuting terms prevent
$(x+y)^p=x^p+y^p$ in general.

When $D$ is centrally finite, the generalized-polynomial argument does not force $f=g$, and the
following classification problem remains.

\begin{problem}\label{prob:centrally-finite-char-two}
Classify the additive solutions of \eqref{eq:main} when $D$ is a centrally finite division ring of
characteristic two and $n\geq3$.
\end{problem}

The symmetry of \eqref{eq:main} gives a preliminary reduction for this problem.

\begin{proposition}\label{P:symmetric}
Let $\charac D = 2$ and let $f,g : D \to D$ be additive maps satisfying \eqref{eq:main}, with
$c = f(1) = g(1)$. Then $(g,f)$ is also a solution of \eqref{eq:main}; the map $u=f+g$ satisfies
\eqref{eq:one-map-reduced}; and $u=P_c$ as a function on $D$, where $P_c$ is defined by
\eqref{eq:Pc-def}.
\end{proposition}

\begin{proof}
By \eqref{eq:g-from-f-char2}, $g(x)=x^nf(x^{-1})x$, and therefore
\[
g(x)x^{-1}+x^nf(x^{-1})=2x^nf(x^{-1})=0.
\]
Thus $(g,f)$ is a solution. Since solutions form an additive group, $(u,u)$ with $u=f+g$ is a
solution, which is precisely \eqref{eq:one-map-reduced} for $u$. Lemma~\ref{lem:hua-reduction}
identifies $u$ with $P_c$.
\end{proof}

Consequently, the symmetric part $f+g$ of a solution is determined by $c$. A first step toward
Problem~\ref{prob:centrally-finite-char-two} is therefore to determine those $c\in D$ for which
$P_c$ is additive, equivalently those for which the generalized polynomial $\Phi_c$ defined in
\eqref{eq:Phi-def} is an identity of $D$.\\

\noindent \textbf{Declarations}

\medskip
\noindent\textbf{\small Funding:} {\small The research of Mohsen Aliabadi is partially supported by an AMS--Simons Research Enhancement Grant (award no.~SFI-MPS-T-Institutes-00021934).} 

\medskip
\noindent \textbf{\small Conflict of interest:} {\small The author declares that he has no conflict of interest.}

\medskip
\noindent \textbf{\small Data availability statement:} {\small Data sharing is not applicable to this article as no datasets were generated or analyzed during the current study.}

\end{document}